\documentclass[11pt]{article}

\usepackage[applemac]{inputenc}
\usepackage{amsmath}
\usepackage{amssymb}
\usepackage{amsthm}
\usepackage{textcomp}
\usepackage{graphicx}

\newtheorem{defn}{Definition}[section]
\newtheorem{lemma}[defn]{Lemma}
\newtheorem{prop}[defn]{Proposition}
\newtheorem{theo}[defn]{Theorem}
\newtheorem{coro}[defn]{Corollary}

\newtheorem{claim}{Claim}
\newtheorem{rk}[defn]{Remark}

\def\Ric{\mathop{\rm Ric}\nolimits}
\def\Hess {\mathop{\rm Hess}\nolimits}
\def\Rm{\mathop{\rm Rm}\nolimits}
\def\det{\mathop{\rm det}\nolimits}

\def\diam{\mathop{\rm diam}\nolimits}
\def\vol{\mathop{\rm vol}\nolimits}
\def\eucl{\mathop{\rm eucl}\nolimits}

\def\vol{\mathop{\rm Vol}\nolimits}

\def\inj{\mathop{\rm inj}\nolimits}

\def\A{\mathop{\rm A}\nolimits}
\def\AVR{\mathop{\rm AVR}\nolimits}
\def\Id{\mathop{\rm Id}\nolimits}
\def\Jac{\mathop{\rm Jac}\nolimits}

\begin{document}
\author{Chih-Wei Chen and Alix Deruelle}
\title{Structure at infinity of expanding gradient Ricci soliton}
\date{}
\maketitle

\begin{abstract}
We study the geometry at infinity of expanding gradient Ricci solitons $(M^n,g,\nabla f)$, $n\geq 3$, with finite asymptotic curvature ratio without curvature sign assumptions. We mainly prove that they have a cone structure at infinity.
\end{abstract}

 \section{Introduction}\label{Intro}
  The central theme of this paper is the study of the geometry at infinity of noncompact Riemannian manifolds. We will focus on the notion of asymptotic cone whose definition is recalled below.
     \begin{defn} Let $(M^n,g)$ be a complete noncompact Riemannian manifold and let $p\in M^n$.
 An \textbf{asymptotic cone} of $(M^n,g)$ at $p$ is a pointed Gromov-Hausdorff limit, provided it exists, of the sequence $(M^n,t_k^{-2}g,p)_k$ where $t_k\rightarrow+\infty$.
 \end{defn}
 
 Usually, the existence of an asymptotic cone is guaranteed by an assumption of nonnegative curvature. More precisely, if $(M^n,g)$ satisfies the nonnegativity assumption $\Ric\geq0$ then the existence of a limit is guaranteed by the Bishop-Gromov theorem and the Gromov's precompactness theorem: see \cite{Pet}.
 We mention two striking results in this direction.
 
 In case of nonnegative sectional curvature, any asymptotic cone of $(M^n,g)$ exists and is unique: it is the metric cone over its ideal boundary $M(\infty)$. Moreover, $M(\infty)$ is an Alexandrov space of curvature bounded below by $1$: see \cite{Gui-Kap}. 
 
   In case of nonnegative Ricci curvature and positive asymptotic volume ratio, i.e., $\lim_{r\rightarrow+\infty}\vol B(p,r)/r^n>0$, Cheeger and Colding proved that any asymptotic cone is a metric cone $C(X)$ over a length space $X$ of diameter not greater than $\pi$: see \cite{Che-Col-War}.
 Nonetheless, even in this case, uniqueness is not ensured: see Perelman \cite{Per-Con}.
 
In this paper, we consider the existence of asymptotic cone on \textit{expanding gradient Ricci solitons} where no curvature sign assumption is made. Instead, we require the finiteness of the \textit{asymptotic curvature ratio}. Such situation has already been investigated by \cite{Che-Chi} in the case of expanding gradient Ricci soliton with vanishing asymptotic curvature ratio.

Recall that the \textbf{asymptotic curvature ratio} of a complete noncompact Riemannian manifold $(M^n,g)$ is defined by 
\begin{eqnarray*}
\A(g):=\limsup_{r_p(x)\rightarrow+\infty}r_p(x)^2\arrowvert\Rm(g)(x)\arrowvert.
\end{eqnarray*}
Note that it is well-defined since it does not depend on the reference point $p\in M^n$. Moreover, it is invariant under scalings. This geometric invariant has generated a lot of interest: see for example \cite{BKN}, \cite{Pet-Tus}, \cite{Lot-She}, \cite{Lot-Vol}, \cite{Esc-Sch} for a static study of the asymptotic curvature ratio and \cite{Ben}, \cite{Ham-For},\cite{Per-Ent}, \cite{Cho-Lu}, \cite{Che-Chi} linking this invariant with the Ricci flow. Note also that Gromov \cite{Gro-Bou} and Lott-Shen \cite{Lot-She} have shown that any paracompact manifold can support a complete metric $g$ with finite $\A(g)$.\\
 
Now, we recall that an \textbf{expanding  gradient Ricci soliton} is a triple \\
$(M^n,g,\nabla f)$ where $(M^n,g)$ is a Riemannian manifold and $f$ is a smooth function on $M^n$  such that 
\begin{eqnarray}
\Ric+\frac{g}{2}=\Hess (f).\label{eq:1}
\end{eqnarray}

It is said \textbf{complete} if $(M^n,g)$ is complete and if the vector field $\nabla f$ is complete.
By \cite{Zha-Zhu}, the completeness of $(M^n,g)$ suffices to ensure the completeness of $\nabla f$.\\
In case of completeness, the associated Ricci flow is defined on $(-1,+\infty)$ by
\begin{eqnarray*}
g(\tau)=(1+\tau)\phi_{\tau}^*g,
\end{eqnarray*}
where $(\phi_{\tau})_{\tau}$ is the $1$-parameter family of diffeomorphisms generated by $\nabla(-f)/{(1+\tau)}$.
A canonical example is the Gaussian soliton $(\mathbb{R}^n,\eucl,\arrowvert x\arrowvert^2/4)$. Along the paper, it is essential to keep this example in mind to get a geometric feeling of the proofs: see \cite{Cao-Lim} for other examples of expanding gradient Ricci solitons.\\

The main result of this paper is the following theorem.

\begin{theo}\label{conv}[Cone structure at infinity]

Let $(M^n,g,\nabla f)$, $n\geq 3$, be a complete expanding gradient Ricci soliton with finite $\A(g)$.

For any sequence $(t_k)_k$ tending to $+\infty$ and $p\in M^n$,  $(M^n,t_k^{-2}g,p)_k$  Gromov-Hausdorff subconverges to a metric cone $(C(S_{\infty}),d_{\infty},x_{\infty})$ over a compact length space $S_{\infty}$.

Moreover, 

1) $C(S_{\infty})\setminus\{x_{\infty}\}$ is a smooth manifold with a $C^{1,\alpha}$ metric $g_{\infty}$ compatible with $d_{\infty}$ and the convergence is $C^{1,\alpha}$ outside the apex $x_{\infty}$.

2) $(S_{\infty},g_{S_{\infty}})$ where $g_{S_{\infty}}$ is the metric induced by $g_{\infty}$ on $S_{\infty}$,   is the $C^{1,\alpha}$ limit of the rescaled levels of the potential function $f$,\\ $(f^{-1}(t_k^2/4),t_k^{-2}g_{t_k^2/4})$ where $g_{t^2/4}$ is the metric induced by $g$ on $f^{-1}(t^2/4)$.

Finally, we can ensure that
\begin{eqnarray}
\arrowvert K_{g_{S_{\infty}}}-1\arrowvert&\leq&\A(g),\quad\mbox{in the Alexandrov sense},\label{eq:alex}\\
\frac{\vol(S_{\infty},g_{S_{\infty}})}{n}&=&\lim_{r\rightarrow+\infty}\frac{\vol B(q,r)}{r^n},\quad \forall q\in M^n.\label{eq:vol}
\end{eqnarray}\\

\end{theo}

As a direct consequence of theorem \ref{conv}, we get uniqueness of the asymptotic cone in case of vanishing asymptotic curvature ratio:

\begin{coro}[Asymptotically flatness]\label{flat}
Let $(M^n,g,\nabla f)$, $n\geq 3$, be a complete expanding gradient Ricci soliton.
Assume
\begin{eqnarray*}
\A(g)=0.
\end{eqnarray*}
Then, with the notations of theorem \ref{conv},
\begin{eqnarray*}
(S_{\infty},g_{S_{\infty}})=\amalg_{i\in I}(\mathbb{S}^{n-1}/ \Gamma_i,g_{std})\quad\mbox{and}\quad(C(S_{\infty}),d_{\infty},x_{\infty})=(C(S_{\infty}), \eucl, 0)
\end{eqnarray*}
where $\Gamma_i$ are finite groups of Euclidean isometries and $\arrowvert I\arrowvert$ is the (finite) number of ends of $M^n$. 

Moreover, for $p\in M^n$,
\begin{eqnarray*}
\sum_{i\in I}\frac{\omega_n}{\arrowvert\Gamma_i\arrowvert}=\lim_{r\rightarrow+\infty}\frac{\vol B(p,r)}{r^n},
\end{eqnarray*}
where $\omega_n$ is the volume of the unit Euclidean ball.
\end{coro}

\begin{rk}
It is not known to the authors if we still have uniqueness in case of finite (or small) asymptotic curvature ratio.
\end{rk}

\begin{rk}
Another consequence of theorem \ref{conv} is to provide examples of expanding gradient Ricci soliton coming out from metric cones. Indeed, under assumptions and notations of theorem \ref{conv}, since $f$ is proper (lemma \ref{distance}), take any $p$ such that $\nabla f (p)=0$. Then, the pointed sequence $(M^n,t_k^{-2}g,p)_k$ is isometric to the pointed sequence $(M^n,g(t_k^{-2}-1),p)$ since the $1$-parameter family of diffeomorphisms generated by $-\nabla f/{(1+\tau)}$, for $\tau \in (-1,+\infty)$, fixes $p$. Therefore, by theorem \ref{conv}, up to a subsequence, such an expanding gradient Ricci soliton comes out from a metric cone. A similar situation has already been encountered in the case of Riemannian manifolds with nonnegative curvature operator and positive asymptotic volume ratio: see \cite{Sch-Sim}.
\end{rk}

\begin{rk}
The major difficulty to prove theorem \ref{conv} is to ensure the existence of an asymptotic cone because no assumption of curvature sign is assumed.
For this purpose, we have to control the growth of the metric balls of such an expanding gradient Ricci soliton: see theorem \ref{Gro-Haus}.
\end{rk}

In the second part of the paper, we will deal with the number of ends of an expanding gradient Ricci soliton. To this aim, we follow closely the work of Munteanu and Sesum \cite{Mun-Ses} on shrinking and steady gradient Ricci solitons. Therefore, we will be sketchy: see \cite{Mun-Ses} for further details and comments. The main result is the following:

\begin{theo}\label{end}
Let $(M^n,g,\nabla f)$, $n\geq 3$, be a complete expanding gradient Ricci soliton such that
$$\inf_{M^n}R>-\frac{n}{2}+1,$$
Then $(M^n,g)$ has one end and it is nonparabolic.
\end{theo}   

Finally,  the asymptotic volume ratio $\AVR(g):=\lim_{r\rightarrow +\infty}\vol B(p,r)/r^n$ is well-defined and positive in case of finite asymptotic curvature ratio $\A(g)$: see lemma \ref{distance}. How can we link these two invariants in a global inequality?
This is the purpose of section 5. For example, we get the

\begin{prop}\label{geo-ineq}

Let $(M^{n},g,\nabla f)$, $n\geq 3$, be a complete expanding gradient Ricci soliton.

Assume $\A(g)<\epsilon$, where $\epsilon$ is a universal constant small enough (less than $3/5$).

Then,
\begin{eqnarray}
\sum_{i\in I}\frac{\omega_n}{\arrowvert \Gamma_i\arrowvert}\frac{1}{(1+\A(g))^{\frac{n-1}{2}}}\leq\AVR(g)\leq \sum_{i\in I}\frac{\omega_n}{\arrowvert \Gamma_i\arrowvert}\frac{1}{(1-\A(g))^{\frac{n-1}{2}}},
\end{eqnarray}
where $\arrowvert I\arrowvert$ is the (finite) number of ends of $M^n$ and $\arrowvert \Gamma_i\arrowvert$ is the order of the fundamental group of the i-th end of $M^n$. \\
\end{prop}

\textbf{Organization}. In section 2, we study the geometry of the levels of the potential function. From this, we get some information about the topology at infinity of expanding gradient Ricci soliton. In section 3, we prove theorem \ref{conv}. In section 4, we give a short proof of theorem \ref{end}. In section 5, we establish geometric inequalities involving the asymptotic curvature ratio $\A(g)$ and the asymptotic volume ratio $\AVR(g)$ of an expanding gradient Ricci soliton. The last two sections do not depend on theorem \ref{conv}. They can be read independently of section 3.\\

\textbf{Acknowledgements.} 
The authors would like to thank Gilles Carron for the proof of theorem \ref{Gro-Haus}.
The first author appreciates his advisor Gérard Besson for his constant encouragement.
The second author would like to thank his advisor Laurent Bessières for his constant support and his precious remarks.

\section{Geometry and topology of the level sets of $f$}\label{section1}

In this section, we consider a complete expanding gradient Ricci soliton (EGS) $(M^n,g,\nabla f)$ satisfying one of the two following basic assumptions: \\

 Assumption $1$ $(A1)$: 
 \begin{eqnarray*}
 \Ric\geq -\frac{C}{1+r_p^2}g, 
\end{eqnarray*}
for some nonnegative constant $C$ and some $p\in M^n$ where $r_p$ denotes the distance function from the point $p$.\\
 
 Assumption 2 $(A2)$:
 \begin{eqnarray*}
\arrowvert\Ric\arrowvert\leq \frac{C}{1+r_p^2},
\end{eqnarray*}
 for some nonnegative constant $C$ and some $p\in M^n$.
 
Of course, (A$2$) implies (A$1$) and finite asymptotic curvature ratio implies (A$2$).

We recall the basic differential equations satisfied by an expanding gradient Ricci soliton \cite{Ben}.

\begin{lemma}\label{id}
Let $(M^n,g,\nabla f)$ be a complete EGS. Then:
\begin{eqnarray}
\Delta f &=& R+\frac{n}{2}, \label{eq:2} \\
\nabla R+ 2\Ric(\nabla f)&=&0, \label{eq:3} \\
\arrowvert \nabla f \arrowvert^2+R&=&f+Cst. \label{eq:4}
\end{eqnarray}
\end{lemma}

In the following, we will assume w.l.o.g. $Cst=0$.

\begin{lemma}\label{distance}[Growth of the potential function]

Let $(M^n,g,\nabla f)$ be a complete EGS and $p\in M^n$. \\

1) If $R\geq -C$ where $C\geq 0$, then $f(x)+C\leq (r_p(x)/2+\sqrt{f(p)+C})^2$.\\

2) Under $(A1)$   then 
\begin{eqnarray*}
f(x)\geq \frac{r_p(x)^2}{4} -\frac{\pi}{2} Cr_p(x)+f(p)-\arrowvert\nabla f(p)\arrowvert.
\end{eqnarray*}

In particular, under $(A1)$, $ f$ behaves at infinity like $\frac{r_p(x)^2}{4}$.\\

3) If $\Ric \geq 0$ then 
\begin{eqnarray*}
\min_{M^n}f+\frac{r_p(x)^2}{4}\leq f(x)\leq (\frac{r_p(x)}{2}+\sqrt{\min_{M^n}f})^2.
\end{eqnarray*}
\end{lemma}

\textbf{Proof}.
Let $p\in M^n$. Let $x\in M^n$ and $\gamma:[0,r_p(x)]\rightarrow M^n$ be a geodesic from $p$ to $x$.
If $R\geq -C$ then \eqref{eq:4} gives $(2\sqrt{f(\gamma(t))+C})'\leq 1.$
Therefore, after integration , we get $2\sqrt{f(x)+C}-2\sqrt{f(p)+C}\leq r_p(x).$
The result follows easily.\\
To get a lower bound for $f$, we apply the Taylor-Young integral formula to $f\circ\gamma$:
\begin{eqnarray*}
f(x)=f(p)+d_pf(\gamma'(0))+\int_{0}^{r_p(x)}(r_p(x)-t)\Hess f(\gamma'(t),\gamma'(t))dt.
\end{eqnarray*}
Using \eqref{eq:1} and $(A1)$, we get
\begin{eqnarray*}
f(x)\geq f(p)-\arrowvert\nabla f(p)\arrowvert+\frac{r_p(x)^2}{4}-C\int_{0}^{r_p(x)}\frac{r_p(x)-t}{1+t^2}dt.
\end{eqnarray*}
Hence the desired inequality.
To prove the last statement, note that under 3), $f$ is a strict convex function for $\Hess f\geq \frac{g}{2}$. Moreover, by $(A1)$, $f$ is a proper function ($C=0$). Therefore $f$ attains its minimum at a unique point $p_0\in M^n$. Now, it suffices to apply the previous results to $f$ and $p_0$.

 \hfill$\square$\\

\begin{rk}
  Note that the bound on the Ricci curvature is not optimal to get a growth of type $r_p(x)^2/4$.
\end{rk}

Under assumptions of lemma \ref{distance} and  using \eqref{eq:4}, the levels of $f$, $M_t:=f^{-1}(t)$, are well-defined compact hypersurfaces for $t>0$ large enough, in particular, they have a finite number of connected components. We will also denote the sublevels (resp. superlevels) of $f$ by $M_{\leq t}:=f^{-1}(]-\infty,t])$ (resp. $M_{\geq t}:=f^{-1}([t,+\infty[)$). $g_t$ will stand for the metric induced on $M_t$ by the ambient metric $g$. 

The next lemma is concerned with the volume of the sublevels of $f$. In case of nonnegative scalar curvature, this lemma has already been proved by several authors: see \cite{Car-Ni}, \cite{Zha-Shi} for instance. See also \cite{Che-Chi}.

\begin{lemma}\label{vol-sub}[Volume of sub-levels of $f$]

Let $(M^n,g,\nabla f)$ be a complete EGS satisfying $(A1)$.

Then, for $1<<t_0\leq t$,
\begin{eqnarray}
\frac{\vol M_{\leq t}}{\vol M_{\leq t_0}}\geq \left(\frac{t+nC} {t_0+nC}\right)^{(n/2-nC)}.
\end{eqnarray}

Moreover, if $R\geq 0$ then $\displaystyle t\rightarrow \frac{\vol M_{\leq t}}{t^{n/2}}$ is a nondecreasing function for $t$ large enough.
\end{lemma}

\textbf{Proof}.
By assumption and \eqref{eq:2}, $-nC\leq R=\Delta f-n/2$. After integrating these inequalities over $M_{\leq t}$, using \eqref{eq:4} and integration by part, one has

\begin{eqnarray*}
-nC\vol M_{\leq t}\leq\int_{M_t}\arrowvert \nabla f\arrowvert-\frac{n}{2}\vol M_{\leq t}
\leq\sqrt{t-\inf_{M_t}R}\vol M_t-\frac{n}{2}\vol M_{\leq t}.
\end{eqnarray*}

Now,
\begin{eqnarray*}
\frac{d}{dt}\vol M_{\leq t}=\int_{M_t}\frac{d\mu_t}{\arrowvert \nabla f\arrowvert}\geq \frac{\vol M_t}{\sqrt{t-\inf_{M_t}R}}.
\end{eqnarray*} 
Hence,
\begin{eqnarray*}
(\frac{n}{2}-nC)\vol M_{\leq t}\leq (t-\inf_{M_t}R)\frac{d}{dt}\vol M_{\leq t}\leq (t+nC)\frac{d}{dt}\vol M_{\leq t}.
\end{eqnarray*}
 The first inequality follows by integrating this differential inequality.
 
 The last statement is obtained by letting $C=0$ in the previous estimates.
 The lower Ricci bound is only used to ensure the existence of the levels of $f$.

\hfill$\square$\\

We pursue by estimating the volume of the hypersurfaces $M_t$.

\begin{lemma}\label{vol-hyper}[Volume of levels I]

Let $(M^n,g,\nabla f)$ be a complete EGS satisfying $(A1)$.

Then, for $1<<t_0\leq t$, there exists a function $a\in L^1(+\infty)$ such that 
\begin{eqnarray}
\displaystyle \frac{\vol M_t}{\vol M_{t_0}}\geq \exp\left(\int_{t_0}^ta(s)ds\right)\left(\frac{t}{t_0}\right)^{\frac{n-1}{2}}.
\end{eqnarray}

Moreover, if $\Ric\geq 0$ then $\displaystyle t\rightarrow \frac{\vol M_t}{t^{(n-1)/2}}$ is a nondecreasing function for $t>\min_{M^n}f$.

\end{lemma}

\textbf{Proof}.
The second fundamental form of $M_t$ is $\displaystyle h_t=\frac{\Hess f}{\arrowvert \nabla f\arrowvert}.$
Now, 
\begin{eqnarray*}
\frac{d}{dt}\vol M_t=\int_{M_t}\frac{H_t}{\arrowvert \nabla f\arrowvert} d\mu_t,
\end{eqnarray*}
where $H_t$ is the mean curvature of the hypersurface $M_t$.
Hence,
\begin{eqnarray*}
\frac{d}{dt}\vol M_t=\int_{M_t}\frac{R-\Ric(\textbf{n},\textbf{n})+\frac{n-1}{2}}{t-R}d\mu_t\geq \frac{(n-1)\inf_{M_t}\Ric+\frac{n-1}{2}}{t-\inf_{M_t}R}\vol M_t,
\end{eqnarray*}
where $\textbf{n}=\nabla f/\arrowvert \nabla f\arrowvert$.
Therefore,
\begin{eqnarray*}
\ln\frac{\vol M_t}{\vol M_{t_0}}&\geq&\int_{t_0}^t\frac{(n-1)(\inf_{M_s}\Ric+\inf_{M_s}R/{2s})}{s-\inf_{M_s}R}ds+\frac{n-1}{2}\ln\left(\frac{t}{t_0}\right)\\
&=&\int_{t_0}^ta(s)ds+\frac{n-1}{2}\ln\left(\frac{t}{t_0}\right),
\end{eqnarray*}
where $a(s):=\displaystyle{ \frac{(n-1)(\inf_{M_s}\Ric+\inf_{M_s}R/{2s})}{s-\inf_{M_s}R}}$.

In view of the lower Ricci bound and lemma \ref{distance}, one has $a\in L^1(+\infty).$ The desired inequality and the case $C=0$ now follow easily.

\hfill$\square$\\

\begin{lemma}\label{vol-hyperII}[Volume of the levels II]

Let $(M^n,g,\nabla f)$ be a complete EGS satisfying $(A2)$.

Then, for $1<<t_0\leq t$, there exists a function $b\in L^1(+\infty)$ such that 
\begin{eqnarray}
\displaystyle \frac{\vol M_t}{\vol M_{t_0}}\leq \exp\left(\int_{t_0}^tb(s)ds\right)\left(\frac{t}{t_0}\right)^{\frac{n-1}{2}}.
\end{eqnarray}
\end{lemma}

\textbf{Proof}.
The proof goes along the same line of the previous one.
Here, the lower Ricci bound is only used to ensure the existence of the hypersurfaces $M_t$ for $t$ large enough by lemma \ref{distance}.

\hfill$\square$\\
\begin{lemma}\label{diam}[Diameter growth]

Let $(M^n,g,\nabla f)$ be a complete EGS satisfying $(A2)$.

Then, for $1<<t_0\leq t$, there exists a function $c\in L^1(+\infty)$ such that 
\begin{eqnarray}
\frac{\diam(g_t)}{\sqrt{t}}\leq \exp\left(\int_{t_0}^tc(s)ds\right)\frac{\diam(g_{t_0})}{\sqrt{t_0}}.
\end{eqnarray}

Assume only $\Ric\geq 0$ then $\displaystyle t\rightarrow \frac{\diam(g_t)}{\sqrt{t}}$ is a nondecreasing function for $t>\min_{M^n}f$.
\end{lemma}

\textbf{Proof}.
Let $\phi_t$ be the flow associated to the vector field $\nabla f/ \arrowvert \nabla f\arrowvert^2$.
If $v\in TM_{t_0}$, define $V(t):=d\phi_{t-t_0}(v)\in TM_t$ for $1<<t_0\leq t$.
Now,
\begin{eqnarray*}
\frac{d}{dt} g(V(t),V(t))=2\frac{\Hess f(V(t),V(t))}{\arrowvert\nabla f\arrowvert^2}=2\frac{\Ric(V(t),V(t))+g(V(t),V(t))/2}{t-R}.
\end{eqnarray*}
Hence,
\begin{eqnarray*}
\frac{d}{dt} g(V(t),V(t))&\leq&\frac{2\sup_{M_t}\Ric+1}{t-\sup_{M_t}R}g(V(t),V(t))\\
&=&\left(c(t)+\frac{1}{t}\right)g(V(t),V(t)),
\end{eqnarray*}
where $c(t):=\displaystyle \frac{2\sup_{M_t}\Ric+\sup_{M_t}R/t}{t-\sup_{M_t}R}.$

 Therefore,
\begin{eqnarray*}
\ln\left(\frac{g(V(t),V(t))}{g(V(t_0),V(t_0))}\right) \leq \int_{t_0}^tc(s)ds+\ln\left(\frac{t}{t_0}\right).
\end{eqnarray*}
In view of the upper Ricci bound, one can see that $c\in L^1(+\infty).$
The desired estimate is now immediate.

The proof of the last assertion use the same arguments.

\hfill$\square$\\

According to the growth of $f$ in case of finite asymptotic ratio ($\A(g)<+\infty$), the hypersurface $M_{t^2/4}$ "looks like" the geodesic sphere $S_t$ of radius $t$. Therefore, the next lemma deals with curvature bounds of the levels $M_{t^2/4}$ for $t$ large.

\begin{lemma}\label{deriv}
Let $(M^n,g,\nabla f)$ be a complete EGS with finite $\A(g)$.

Then,
\begin{eqnarray}
\limsup_{t\rightarrow+\infty} \arrowvert  \Rm(t^{-2}g_{t^2/4})-\Id_{\Lambda^2}\arrowvert&\leq& \A(g), \label{eq:curvop} \\ 
1-\A(g)\leq \liminf_{t\rightarrow+\infty} K_{t^{-2}g_{t^2/4}}&\leq&\limsup_{t\rightarrow+\infty} K_{t^{-2}g_{t^2/4}}\leq 1+\A(g)\label{eq:alex2},
\end{eqnarray}
In case of nonnegative sectional curvature, one has 
\begin{eqnarray}
1\leq \liminf_{t\rightarrow+\infty} K_{t^{-2}g_{t^2/4}}&\leq&\limsup_{t\rightarrow+\infty} K_{t^{-2}g_{t^2/4}}\leq 1+\A(g).
\end{eqnarray}
\end{lemma}

\textbf{Proof}.\\
Take a look at Gauss equations applied to $M_{t^2/4}$.
\begin{eqnarray*}
 \Rm(g_{t^2/4})(X,Y)&=&\Rm(g)(X,Y)+\det h_{t^2/4}(X,Y)\\
 &=&\Rm(g)(X,Y)+\frac{\det(\Ric+g/2)}{\arrowvert \nabla f\arrowvert^2}\big|_{M_{t^2/4}}, 
 \end{eqnarray*}
where $X$ and $Y$ are tangent to $M_{t^2/4}$.
After rescaling the metric $g_{t^2/4}$ by $t^{-2}$, and using the fact that $t^2/\arrowvert\nabla f\arrowvert^2 \big|_{M_{t^2/4}}\rightarrow 0$ as $t\rightarrow +\infty$, we get all the desired inequalities.

\hfill$\square$\\

Using the classification result of \cite{Boh-Wil} together with inequality (\ref{eq:curvop}) of lemma \ref{deriv} , we get the following topological information:

\begin{coro}[Small asymptotic ratio]\label{small}

Let $(M^n,g,\nabla f)$, $n\geq3$ be a complete EGS such that $\A(g)<1$.

Then, outside a compact set $K$, $M^n\setminus K$ is a disjoint union of a finite number of ends, each end being diffeomorphic to $\mathbb{S}^{n-1}/\Gamma\times(0,+\infty)$ where $\mathbb{S}^{n-1}/\Gamma$ is a spherical space form.
\end{coro}

This result can be linked with previous results concerning the topology at infinity of Riemannian manifolds with cone structure at infinity and finite (vanishing) asymptotic curvature ratio: \cite{Pet}, \cite{Lot-She},\cite{Gre-Pet}, \cite{Che-Chi}.\\

 The following lemma establishes some links between the volume growth of metric balls, sublevels and levels:

\begin{lemma}\label{sequence}
Let $(M^n,g,\nabla f)$ be a complete EGS satisfying $(A1)$.

Then for any $q\in M$,
\begin{eqnarray}
\liminf_{r\rightarrow+\infty}\frac{\vol B(q,r)}{r^n}=\liminf_{t\rightarrow+\infty}\frac{\vol M_{\leq t^2/4}}{t^n}\geq \limsup_{t\rightarrow+\infty}\frac{\vol M_{t^2/4}}{nt^{n-1}}>0.\label{eq:volinf}
\end{eqnarray}
Moreover, if we assume $(A2)$ then,
\begin{eqnarray}
\lim_{r\rightarrow+\infty}\frac{\vol B(q,r)}{r^n}=\lim_{t\rightarrow+\infty}\frac{\vol M_{\leq t^2/4}}{t^n}= \lim_{t\rightarrow+\infty}\frac{\vol M_{t^2/4}}{nt^{n-1}}<+\infty.\label{equality-vol}
\end{eqnarray}
\end{lemma}

\textbf{Proof}.\\
    Under $(A1)$, lemma \ref{distance} tells us that the potential function $f$ is equivalent at infinity to $r_p^2/4$ for $p\in M$. So the first equality is clear.
Next, for $1<<t_0\leq t$, by lemma \ref{vol-hyper}
\begin{eqnarray*}
\frac{\vol M_{\leq t^2/4}}{t^n}&\geq&\frac{1}{t^n}\int_{t_0^2/4}^{t^2/4}\frac{\vol M_s}{\sqrt{s+nC}}ds\\
 &\geq&\frac{2^{n-1}}{t^n}\exp\left(\int_{t_0^2/4}^{t^2/4} a(s)ds\right)\frac{\vol M_{t_0^2/4}}{t_0^{n-1}} \int_{t_0^2/4}^{t^2/4}\frac{s^{\frac{n-1}{2}}}{\sqrt{s+nC}}ds.\\
\end{eqnarray*}
So, letting $t\rightarrow+\infty$ in the previous inequality gives for any large $t_0$,
\begin{eqnarray*}
\liminf_{t\rightarrow+\infty}\frac{\vol M_{\leq t^2/4}}{t^n}\geq \exp\left(\int_{t_0^2/4}^{+\infty} a(s)ds\right)\frac{\vol M_{t_0^2/4}}{nt_0^{n-1}}.
\end{eqnarray*}
 Hence inequality (\ref{eq:volinf}) by making $t_0\rightarrow +\infty$.

In case of an upper Ricci bound, using lemma \ref{vol-hyperII}, we get 
\begin{eqnarray*}
\limsup_{r\rightarrow+\infty}\frac{\vol B(q,r)}{r^n}=\limsup_{t\rightarrow+\infty}\frac{\vol M_{\leq t^2/4}}{t^n}\leq \liminf_{t\rightarrow+\infty}\frac{\vol M_{t^2/4}}{nt^{n-1}}<+\infty.
\end{eqnarray*}

Therefore, all the limits exist and are equal. Moreover, they are positive and finite. 

\hfill$\square$\\

\section{Proof of the main theorem \ref{conv}}

\subsection{Gromov-Hausdorff convergence}

Let $(M^n,g,\nabla f)$ be a complete expanding gradient Ricci soliton with finite asymptotic curvature ratio. Let $p\in M^n$ and $(t_k)_k$ any sequence tending to $+\infty$.

 \begin{claim} \textit{The sequence $(M^n,t_k^{-2}g,p)_k$ of pointed Riemannian manifolds contains a convergent subsequence in the pointed Gromov-Hausdorff sense.}
 \end{claim}
  According to the Gromov's precompactness theorem (see chap.10 of \cite{Pet} for instance), the claim follows from the following theorem, which is essentially due to Gilles Carron \cite{Che-Chi}.
  
  \begin{theo}\label{Gro-Haus}
  Let $(M^n,g,\nabla f)$ be a complete EGS with finite asymptotic curvature ratio. Then there exist positive constants $c$, $c'$ such that for any $x\in M^n$ and any radius $r>0$,
  \begin{eqnarray*}
cr^n\leq\vol B(x,r)\leq c'r^n.
\end{eqnarray*}
\end{theo} 
 
 The arguments are essentially the same as in \cite{Che-Chi}. We give the proof for completeness.\\
 
 \textbf{Proof.}\\
 Along the proof, $c$, $c'$ will denote constants independent of $t$ which van vary from line to line.\\
 
 \textbf{Step1}.
 We begin by the 
 \begin{lemma}\label{volboule1}
  There exists $R>0$ and $c>0$ such that for $r_p(x)\geq R$,
 \begin{eqnarray*}
c(r_p(x)/2)^n\leq \vol B(x,r_p(x)/2).
\end{eqnarray*}
\end{lemma}

 \textbf{Proof.}
Indeed, by lemmas  \ref{vol-hyper}, \ref{diam} and \ref{deriv}, we know that there exists a positive constant $i_0$ such that $\inj (M_{t^2/4},t^{-2}g_{t^2/4})\geq i_0$ for $t\geq t_0>>1$.
Therefore, for $t\geq t_0>>1$ and $x\in M_{t^2/4}$,
\begin{eqnarray}
\vol_{g_{t^2/4}}B_{g_{t^2/4}}(x,i_0t/2)\geq ct^{n-1},\label{eq:volboule}
\end{eqnarray}
for some positive constant $c$.\\
 \begin{claim}  \label{claim2}
\begin{eqnarray*}
\phi_{v}(B_{g_{t^2/4}}(x,i_0t/2))\subset B_g(x,r_p(x)/2),
\end{eqnarray*}
for $v\in[0,\alpha t^2/4]$, where $\alpha$ is a positive constant independent of $t$.
 \end{claim}
 
 \begin{proof}[Proof of Claim \ref{claim2}]
 Let $y\in B_{g_{t^2/4}}(x,i_0t/2)$. By the triangular inequality,
 \begin{eqnarray*}
d_g(x,\phi_v(y))\leq d_g(x,y)+d_g(y,\phi_v(y))\leq i_0t/2+d_g(y,\phi_v(y)).
\end{eqnarray*}
 Thus, it suffices to control the growth of the function $\psi(v):=d_g(y,\phi_v(y))$ for $v\geq 0$. Now, for $t\geq t_0>>1$,
\begin{eqnarray*}
\psi(v)&\leq&\int_{0}^{v}\arrowvert\psi'(s)\arrowvert ds
\leq \int_{0}^{v}\frac{ds}{\arrowvert\nabla f\arrowvert(\phi_s(y))}\\
&=&\int_{0}^{v}\frac{ds}{\sqrt{s+t^2/4-R(\phi_s(y))}}\\
&\leq&\int_{0}^{v}\frac{ds}{\sqrt{s}}=2\sqrt{v}.
\end{eqnarray*}
Therefore,
 \begin{eqnarray*}
d_g(x,\phi_v(y))\leq i_0t/2+2\sqrt{v}\leq(i_0/2+\sqrt{\alpha})t,
\end{eqnarray*}
 for $v\in[0,\alpha t^2/4]$ and any $\alpha>0$.
 The claim now follows by using the growth of the potential function  given by lemma \ref{distance} and choosing a suitable $\alpha$ sufficiently small.
 
  \end{proof}
 
 By Claim \ref{claim2} and the coarea formula, we have
 \begin{eqnarray}
\vol B_g(x,r_p(x)/2)&\geq& \vol\{ (\phi_{v}(B_{g_{t^2/4}}(x,i_0t/2)); v\in[0,\alpha t^2/4])\}\\
&=& \int_{0}^{\alpha t^2/4}\int_{B_{g_{t^2/4}}(x,i_0t/2)}\frac{\arrowvert\Jac(\phi_v)\arrowvert}{\arrowvert\nabla f\arrowvert}dA_{t^2/4}dv. \label{eq:volboule2}
\end{eqnarray}
Now, for $t\geq t_0>>1$ and  $y\in M_{t^2/4}$, the maps $s\rightarrow\phi_s(y)$ for $s\geq 0$ are expanding since, as in the proof of lemma \ref{diam}, for $v\in TM_{t^2/4}$,
\begin{eqnarray*}
\frac{d}{ds} g(d\phi_s(v),d\phi_s(v))=2\frac{\Ric(d\phi_s(v),d\phi_s(v))+g(d\phi_s(v),d\phi_s(v))/2}{\arrowvert\nabla f\arrowvert^2}\geq0.
\end{eqnarray*}

Combining this fact with inequalities (\ref{eq:volboule}) and (\ref{eq:volboule2}), we get
\begin{eqnarray*}
\vol B_g(x,r_p(x)/2)&\geq&\frac{\alpha t^2/4}{\max_{M_{t^2/4}}\arrowvert\nabla f\arrowvert}\vol_{g_{t^2/4}}B_{g_{t^2/4}}(x,i_0t/2)\\
&\geq& ct^{n}.
\end{eqnarray*}
 This ends the proof of lemma \ref{volboule1}.
 
 \hfill$\square$\\

\textbf{Step 2.}
For $r_p(x)\geq R$, we know that $\Ric\geq -C^2/{r_p(x)^2}$ on the ball $B(x,r_p(x)/2)$ for $C$ independent of $x$ since $\A(g)<+\infty$.
Therefore, by the Bishop-Gromov theorem, for $r\leq r_p(x)/2$,
\begin{eqnarray*}
\vol B(x,r)&\geq&\frac{\vol(n, -(C/r_p(x))^2,r)}{\vol (n,-(C/r_p(x))^2,r_p(x)/2)}\vol B(x,r_p(x)/2),
\end{eqnarray*}
where $\vol(n,-k^2,r)$ denotes the volume of a ball of radius $r$ in the $n$-dimensional hyperbolic space of constant curvature $-k^2$.

Now,
\begin{eqnarray*}
\vol (n,-k^2,r)&=&\vol(\mathbb{S}^{n-1})\int_{0}^r\left(\frac{\sinh(kt)}{k}\right)^{n-1}dt\\
&\geq&\frac{\vol(\mathbb{S}^{n-1})}{n}r^n,
\end{eqnarray*}
and
\begin{eqnarray*}
\vol (n,-(C/r_p(x))^2,r_p(x)/2)=\left(\frac{\vol(\mathbb{S}^{n-1})}{C^n}\int_{0}^{C/2}\sinh(u)^{n-1}du\right)r_p(x)^n.
\end{eqnarray*}

To sum it up, we get by lemma \ref{volboule1},
\begin{eqnarray*}
\vol B(x,r)\geq cr^n, \quad \mbox{for $r\leq r_p(x)/2$ and $r_p(x)\geq R$}.
\end{eqnarray*}

Next, inside the compact ball $B(p,R)$, we will get a similar lower bound for $r\leq R$ because of the continuity of the function $(x,r)\rightarrow \vol B(x,r)/r^n$.

Now, for any $r>0$, choose $x\in M^n$ such that $r_p(x)=r/2$. Then $B(x,r/4)\subset B(p,r)$ and the lower bound follows for any ball centered at $p$.

Finally, for any $x\in M^n$ and radius $r$ satisfying $r\geq 2r_p(x)$, $B(p,r/2)\subset B(x,r)$. Hence the lower bound for any balls $B(x,r)$ for $r\leq r_p(x)/2$ and $r\geq 2r_p(x)$.
Since for $r\in[r_p(x)/2,2r_p(x)]$, $\vol B(x,r)\geq \vol B(x,r_p(x)/2)\geq c(r_p(x)/2)^n\geq c(r/4)^n$, the proof of the lower bound is finished.

\textbf{Step 3.}

To get an upper bound, once again, by the Bishop theorem,
\begin{eqnarray*}
\vol B(x,r)\leq\vol(n,-(C/r_p(x))^2,r)\leq c'r^n, \quad \mbox{for $r\leq r_p(x)/2$},
\end{eqnarray*}
where $c':=\vol(\mathbb{S}^{n-1})\max_{u\in[0,C/2]}u^{-n}\int_{0}^{u}(\sinh(s))^{n-1}ds.$

For $r\geq r_p(x)/2$, $B(x,r)\subset B(p,3r)$ and $\vol B(p,r)\leq c'r^n$ for $r$ large enough (say $3R/2$) by lemma \ref{sequence}. So, $\vol B(x,r)\leq \vol B(p,3r)\leq c'3^nr^n$ for $r_p(x)\geq R$ and $r\geq r_p(x)/2$.

Invoking again the continuity of the volume ratio on $B(p,R)\times [0,R]$, we end the proof of theorem \ref{Gro-Haus}.

 \hfill$\square$\\

 \subsection{$C^{1,\alpha}$ convergence}
 
  \textbf{Step 1}: $(f^{-1}(t_k^2/4),t_k^{-2}g_{t_k^2/4})$ subconverges in the $C^{1,\alpha}$-topology to a compact smooth manifold with a $C^{1,\alpha}$-metric.\\

Indeed, according to the lemmas \ref{vol-hyper}, \ref{diam} and \ref{deriv}, we are in a position to apply the $C^{1,\alpha}$-compactness theorem \cite{Gre-Wu-Con}, \cite{Pet-Con}, \cite{Kas-Con}, to the sequence $(M_{t_k^2/4},t_k^{-2}g_{{t_k^2}/4})_{k}$. This shows the second part of theorem \ref{conv}. Moreover, by inequality (\ref{eq:alex2}) of lemma \ref{deriv}, we immediately get the estimate (\ref{eq:alex}). Equally, by equality (\ref{equality-vol}) of lemma \ref{sequence}, equality (\ref{eq:vol}) follows. \\
\textbf{Step 2}:

For $0<a<b$, consider the annuli $(M_{at^2/4\leq s\leq bt^2/4},t^{-2}g)=:(M_{a,b}(t),t^{-2}g)$ for positive $t$ . Because of the finiteness of $\A(g)$, it follows that 
\begin{eqnarray*}
\lim_{t\rightarrow+\infty}\sup_{M_{a,b}(t)}\arrowvert\Rm(t^{-2}g)\arrowvert\leq \A(g).
\end{eqnarray*}

Moreover, by a local version of Cheeger's injectivity radius estimate [CGT], lemma \ref{volboule1} and by the finiteness of $\A(g)$, there exists a positive constant $\iota_0$ such that for any $x\in M^n$,
\begin{eqnarray*}
\inj(x,g)\geq\iota_0 r_p(x).
\end{eqnarray*}

Now, consider the sequence of pointed complete Riemannian manifolds \\
$(M^n,t_k^{-2}g,p)_k$.
By theorem \ref{Gro-Haus}, $(M^n,t_k^{-2}g,p)$ Gromov-Hausdorff subconverges to a metric space $(X_{\infty},d_{\infty},x_{\infty})$. 
By a local form of the $C^{1,\alpha}$-\\compactness theorem [BKN] and by the previous annuli estimates, one can deduce that $X_{\infty}/\{x_{\infty}\}$ is a smooth manifold with a $C^{1,\alpha}$ metric compatible with $d_{\infty}$ and that the convergence is $C^{1,\alpha}$ outside the apex $x_{\infty}$.\\
\hfill$\square$\\

 Finally, we prove corollary \ref{flat}. \\

\textbf{Proof of corollary \ref{flat}}.
It is not straightforward since the convergence is only $C^{1,\alpha}$. Still, one can apply the  results of the proof of theorem $78$ of the book \cite{Pet}. We sum up the major steps. On the one hand, one shows that the limit metric $g_{S_{\infty}}$ is weakly Einstein, hence smooth by  elliptic regularity. On the other hand, one sees, in polar coordinates, that the metrics $t^{-2}g_{t^2/4}$ $C^{1,\alpha}$-converges to a constant curvature metric. These facts with theorem \ref{conv} suffice to prove  corollary \ref{flat}. 

\hfill$\square$\\

\section{Ends of expanding gradient Ricci solitons}

As we saw in the last sections, the asymptotic curvature ratio controls the topology of the ends of an expanding gradient Ricci soliton. Nonetheless, it does not give any information on the number of ends. Since the fruitful approach initiated, among others, by P. Li, L.F. Tam and J. Wang, we know that counting ends is related to the study of harmonic functions with finite energy: see \cite{Li-Har} for a nice survey. Munteanu and Sesum \cite{Mun-Ses} have used this method to study the ends of shrinking and steady solitons. We will use closely their arguments in our case.

First, we recall a general result due to Carrillo and Ni \cite{Car-Ni} on the growth of metric balls of an EGS with scalar curvature bounded from below, which is always true: see \cite{Zha-Zhu}, \cite{Pig-Rim-Set}.

\begin{prop}[Carrillo-Ni]\label{Car-Ni}
Let $(M^n,g,\nabla f)$ be a complete EGS.
Then, for any $p\in M^n$ and $r\geq r_0$,
\begin{eqnarray*}
\frac{\vol B(p,r)}{\vol B(p,r_0)}\geq \left(\frac{r-2\sqrt{f(p)-\inf_{M^n}R}}{r_0-2\sqrt{f(p)-\inf_{M^n}R}}\right)^{n+2\inf_{M^n}R}.
\end{eqnarray*}
\end{prop}

In the following theorem, first, we investigate the existence of nontrivial harmonic functions with finite energy. Then, we study the nonparabolicity of the ends of an EGS.
We recall that an end $E$ is said \textbf{nonparabolic} if $E$ has a positive Green's function with Neumann boundary condition.

\begin{theo}\label{int-end}
Let $(M^n,g,\nabla f)$ be a complete EGS.\\
1) Assume $R\geq -\frac{n}{2}+1.$

Then any harmonic function with finite energy is constant.\\
2) Assume $n\geq 3$ and $\inf_{M^n}R> -n/2+1$.

Then any end is nonparabolic.
\end{theo}

\textbf{Proof.}\\
\textit{Proof of 1):}\\
Let $u:M^n\rightarrow \mathbb{R}$ satisfying 
\begin{eqnarray*}
\Delta u=0 \quad\mbox{and}\quad \int_{M^n}\arrowvert\nabla u\arrowvert^2<+\infty.
\end{eqnarray*}

Let $\phi$ be any cut-off function. Then, as in the proof of theorem 4.1 of \cite{Mun-Ses},
after several integrations by parts, and by the Bochner formula applied to $u$, we get
\begin{eqnarray*}
\int_{M^n}(R+\frac{n}{2}-1)\arrowvert\nabla u\arrowvert^2\phi^2+\int_{M^n}\arrowvert\nabla\arrowvert\nabla u\arrowvert\arrowvert^2\phi^2\leq\\
4\int_{M^n}\arrowvert\nabla u\arrowvert^2\arrowvert\nabla \phi\arrowvert^2+3\int_{M^n}\arrowvert\nabla u\arrowvert^2\arrowvert\nabla f\arrowvert\arrowvert\nabla \phi^2\arrowvert.
\end{eqnarray*}

Now, take $\phi$ such that 
\begin{eqnarray*}
\phi(x)=\left\{\begin{array}{rl}
1 & \mbox{on $B(p,r)$}\\

(2r-r_p(x))/r & \mbox{on $B(p,2r)\setminus B(p,r)$}\\

0 & \mbox{on $M^n\setminus B(p,2r)$}\end{array}\right.
\end{eqnarray*}
for $p\in M^n$ fixed and $r>0$.
Note that $\arrowvert\nabla u\arrowvert\leq1/r$. 

Moreover, by lemma \ref{distance},
$\arrowvert\nabla f\arrowvert\leq r+\sqrt{f(p)-\inf_{M^n}R}$ on $B(p,2r)$.

Therefore, as $r\rightarrow+\infty$, we get by the previous inequality and the assumption on the scalar curvature:
\begin{eqnarray*}
\arrowvert\nabla\arrowvert\nabla u\arrowvert\arrowvert=(R+\frac{n}{2}-1)\arrowvert\nabla u\arrowvert^2=0.
\end{eqnarray*}
In any cases, $\arrowvert\nabla u\arrowvert$ is constant and by proposition \ref{Car-Ni}, the volume of $M^n$ is infinite, so that $\arrowvert\nabla u\arrowvert=0$.\\

\textit{Proof of 2):}\\
To detect nonparabolicity, we will use a criterion given by lemma $3.6$ of \cite{Li-Pre}:

\begin{lemma}[Li]\label{Li-crit}
An end $E$ with respect to the compact ball $B(p,r_0)$ is nonparabolic if and only if the sequence of harmonic functions $u_i$ defined on $E(p,r_i):=E\cap B(p,r_i)$ for $r_i\rightarrow +\infty$, satisfying $u_i=1$ on $\partial E$, $u_i=0$ on $\partial B(p,r_i)\cap E$ converges to a nonconstant harmonic function on $E$.
\end{lemma}

We will also use a weighted Poincaré inequality in the sense of \cite{Li-Wan-Poi}:\\
If  $\inf_{M^n}R> -n/2+1$ on a complete EGS $(M^n,g,\nabla f)$ with $n\geq 3$, then $(M^n,g)$ satisfies:
\begin{eqnarray}
\int_{M^n}\frac{\inf_{M^n}R+n/2-1}{2(f+n/2-1)}\phi^2\leq\int_{M^n}\arrowvert\nabla \phi\arrowvert^2,\label{weight}
\end{eqnarray}
for any $\phi\in C_c^{\infty}(M^n)$.\\

A similar inequality is proved in \cite{Mun-Ses} for shrinking gradient Ricci solitons. The proof is the same with minor modifications: see the remark after proposition $3.8$ of \cite{Mun-Ses}. \\

We are in a position to end the proof of 2):\\
Let $E$ be an end with respect to the ball $B(p,r_0)$.
Let $(u_i)_i$ a sequence constructed as in lemma \ref{Li-crit} and assume on the contrary that $(u_i)_i$ converges to the constant function $1$ on $E$. Now, we mimic the proof of lemma 3.10 of \cite{Li-Pre} to get a contradiction.
Let $r_0<r$ such that $E(p,r_0)\neq\emptyset$ and let $\phi$ be a nonnegative cut-off function satisfying
\begin{eqnarray*}
\phi(x)=\left\{\begin{array}{rl}
0 & \mbox{on $\partial E$}\\

1 & \mbox{on $E(p,r)\setminus E(p,r_0)$}\\

\arrowvert\nabla\phi\arrowvert\leq C& \mbox{on $E$}\end{array}\right .
\end{eqnarray*}

Apply inequality (\ref{weight}) to the test function $\phi u_i$. As $u_i$ is harmonic, we get
\begin{eqnarray*}
\int_{E(p,r_1)\setminus E(p,r_0)}\frac{\inf_{M^n}R+n/2-1}{2(f+n/2-1)}u_i^2\leq C\int_{E(p,r_0)}u_i^2,
\end{eqnarray*}
 for any $r_0<r_1<r_i$, where $C$ is a constant independent of $i$. As $(u_i)_i$ converges uniformly on compact sets to $1$, we have
 \begin{eqnarray*}
\int_{E(p,r_1)\setminus E(p,r_0)}\frac{\inf_{M^n}R+n/2-1}{2(f+n/2-1)}\leq C\vol E(p,r_0),
\end{eqnarray*}
for any $r_0<r_1$. Now, by the growth of $f$ given by lemma \ref{distance}, we deduce the following growth estimate:
\begin{eqnarray*}
\vol (E(p,r_1)\setminus E(p,r_0))\leq C\vol E(p,r_0) r_1^2.
\end{eqnarray*}
This is a contradiction with the growth estimate of proposition \ref{Car-Ni} since $n+2\inf_{M^n}R>2$ here. (Indeed, the result of proposition \ref{Car-Ni} can be localized to any end of $M^n$.)

\hfill$\square$\\

\begin{rk}
Note that when $n=2$, the condition in 1) of theorem \ref{int-end} is just $Ric\geq 0$. This was already known in a general setting by \cite{Yau-Fun}. 
\end{rk}
\begin{rk}
2) of theorem \ref{int-end} is sharp in the dimension since when $n=2$, the assumption would be $\inf_{M^n} R>0$. This is impossible because $\inf_{M^n}R\leq 0$ for any EGS by \cite{Pig-Rim-Set}.
\end{rk}
As in corollary $3.7$ of \cite{Mun-Ses}, we deduce from 1) of theorem \ref{int-end}, the 

\begin{coro}\label{nonpara}
Let $(M^n,g,\nabla f)$ be a complete EGS with $R\geq -n/2+1$.
Then it has at most one nonparabolic end.
\end{coro}

\textbf{Proof of theorem \ref{end}.}\\
According to 2) of theorem \ref{int-end}, any end of such an EGS is nonparabolic. Now, the result follows from corollary \ref{nonpara}.

\hfill$\square$\\

\begin{rk}
The bound in theorem \ref{end} is sharp since the metric product $\mathbb{R}^2\times N^{n-2}$, for $n\geq 4$, where $N^{n-2}$ is any compact Einstein manifold of constant scalar curvature $-(n-2)/2$ is parabolic and its scalar curvature is exactly $-n/2+1$. In dimension $3$, $\mathbb{R}\times \Sigma^2_g$ where $\Sigma^2_g$ is a surface of genus $g\geq 2$ of constant scalar curvature $-1$ is parabolic. Nonetheless, its scalar curvature is $-1$, which is less than $-1/2$. In fact, EGS with constant scalar curvature in dimension $3$ have been classified by Petersen and Wylie: see \cite{Pet-Wyl-Cla}. The possible values belong to $\{-3/2,-1,0\}$. To what extent the value $-1/2$ is a critical value for EGS in dimension $3$?
\end{rk}

\section{Volume monotonicity and geometric inequalities}

\subsection{Volume monotonicity}

We begin by stating volume monotonicity results:

Combining the  lemma \ref{sequence} with the monotonicity results of lemma \ref{vol-sub} and \ref{vol-hyper}, we get the 

\begin{coro}\label{vol-est}
Let $(M^n,g,\nabla f)$ be a complete EGS.

1) Assume $(A2)$ and $R\geq 0$.\\

Then, $ t\rightarrow \vol M_{\leq t^2/4}/{t^n}$ is nondecreasing and
\begin{eqnarray}
0<\AVR(g):=\lim_{r\rightarrow+\infty}\frac{\vol B(p,r)}{r^n}=\lim_{t\rightarrow+\infty}\frac{\vol M_{\leq t^2/4}}{t^n}<+\infty.
\end{eqnarray}

2) Only assume $\Ric\geq 0$ then
\begin{eqnarray}
0<\lim_{t\rightarrow+\infty}\frac{\vol M_{t^2/4}}{nt^{n-1}}\leq\AVR(g)=\lim_{t\rightarrow+\infty}\frac{\vol M_{\leq t^2/4}}{t^n}
\end{eqnarray}
with equality if, for instance, the scalar curvature is bounded from above.
\end{coro}

The following corollary was already known in a more general context by \cite{BKN}.
\begin{coro}
Let $(M^n,g,\nabla f)$ be a complete EGS.
Assume
\begin{eqnarray*}
\Ric\geq 0\quad\mbox{and}\quad\A(g)=0.
\end{eqnarray*}
Then $(M^n,g,\nabla f)$ is isometric to the Gaussian expanding soliton.
\end{coro}

\textbf{Proof.}

As in \cite{Car-Ni} and in the proof of lemma \ref{distance}, in case of $\Ric\geq 0$, $f$ is a proper strictly convex function, hence $M^n$ is diffeomorphic to $\mathbb{R}^n$.
Therefore, by corollary \ref{flat}, corollary \ref{vol-est}, 
\begin{eqnarray*}
\omega_n = \lim_{r\rightarrow+\infty}\frac{\vol B(p,r)}{r^n}=\AVR(g).
\end{eqnarray*}

The result now follows by the rigidity part of the Bishop-Gromov theorem.

\hfill$\square$\\

\subsection{Geometric Inequalities}

 Here, we link $\A(g)$ and $\AVR(g)$ in a global inequality. 
An easy way is to use the Gauss-Bonnet theorem (see \cite{Ber-Pan}) which is only valid for a global odd dimension.
\begin{prop}

Let $(M^{n},g,\nabla f)$ be a complete EGS with $n$ odd.
Assume
\begin{eqnarray}
 \Ric\geq 0\quad\mbox{and}\quad\A(g)<+\infty.
\end{eqnarray}
Then
\begin{eqnarray}
\frac{\omega_n}{(1+\A(g))^{\frac{n-1}{2}}}\leq\AVR(g).
\end{eqnarray}

\end{prop}

\textbf{Proof}.
As we have already seen, $M_{t^2/4}$ is diffeomorphic to a $(n-1)$-sphere for $t>\min_{M^n}$, since $\Ric\geq 0$.
Therefore, apply the Gauss-Bonnet formula to the $(n-1)$-sphere $M_{t^2/4}$:
\begin{eqnarray*}
2=\chi(M_{t^2/4})=\frac{2}{\vol(\mathbb{S}^{n-1})}\int_{M_{t^2/4}}\textbf{K},
\end{eqnarray*}
where 
\begin{eqnarray*}
\textbf{K}=
\frac{1}{(n-1)!}\sum_{i_1<...<i_{n-1}}\epsilon_{i_1,...,i_{n-1}}\Rm_{i_1,i_2}\wedge...\wedge\Rm_{i_{n-2},i_{n-1}},
\end{eqnarray*}
and $\Rm$ is the curvature form of the metric $t^{-2}g_{t^2/4}$ and $\epsilon_{i_1,...i_{n-1}}$ is the signature of the permutation $(i_1,...,i_{n-1})$.

As $\vol(\mathbb{S}^{n-1})=n\omega_n$ and making $t_k\rightarrow +\infty$ where $(t_k)_k$ is as in theorem \ref{conv}, one has by theorem \ref{conv} and corollary \ref{vol-est},

\begin{eqnarray*}
n\omega_n\leq (1+\A(g))^{\frac{n-1}{2}}\vol(S_{\infty},g_{S_{\infty}})\leq (1+\A(g))^{\frac{n-1}{2}}(n\AVR(g)).
\end{eqnarray*}

\hfill$\square$\\

In case $n$ is not necessarily odd, we still get such an inequality for a small asymptotic curvature ratio: \\

\textbf{Proof of proposition \ref{geo-ineq}}.
Along the proof, we will assume that $M^n$ has only one end. In case of more than one end, the following arguments can be applied to each end and the proposition is established by summing over the ends.

Therefore, consider the connected compact hypersurfaces $(M_{t^2/4},t^{-2}g_{t^2/4})$ for $t$ large enough.
By lemma \ref{deriv}, 
\begin{eqnarray*}
1-\A(g)\leq \liminf_{t\rightarrow+\infty} K_{t^{-2}g_{t^2/4}}&\leq&\limsup_{t\rightarrow+\infty} K_{t^{-2}g_{t^2/4}}\leq 1+\A(g).
\end{eqnarray*}
If $\A(g)$ is less than $1$, then, by Myers' theorem,  $\Gamma=\pi_1(M_{t^2/4})$ is finite and by the Bishop theorem, we get the second inequality.

If we consider the Riemannian finite universal coverings $(\widetilde{M_{t^2/4}},\widetilde{t^{-2}g_{t^2/4}})$ of these hypersurfaces, the previous curvature inequalities will be preserved and if $\A(g)$ is small enough (less than $3/5$) then 
\begin{eqnarray*}
\frac{1}{4}<\frac{1-\A(g)}{1+\A(g)}.
\end{eqnarray*}
Therefore, by  Klingenberg's result, \cite{Bre-Sch} for a survey on sphere theorems, the injectivity radius of $(\widetilde{M_{t^2/4}},\widetilde{t^{-2}g_{t^2/4}})$ will be asymptotically greater than \\
$\pi/\sqrt{1+\A(g)}$. Thus,
\begin{eqnarray*}
\frac{n\omega_n}{(1+\A(g))^{\frac{n-1}{2}}}=\frac{\vol(\mathbb{S}^{n-1})}{(1+\A(g))^{\frac{n-1}{2}}} &\leq&\lim_{t\rightarrow +\infty}\vol(\widetilde{M_{t^2/4}},\widetilde{t^{-2}g_{t^2/4}})\\
&=&\arrowvert \Gamma\arrowvert \lim_{t\rightarrow +\infty}\vol(M_{t^2/4},t^{-2}g_{t^2/4})\\&\leq&\arrowvert \Gamma\arrowvert n\AVR(g).
\end{eqnarray*}

\hfill$\square$\\

As a direct consequence, we get the

\begin{coro}
Let $(M^{n},g,\nabla f)$ be a complete EGS with $n\geq 3$.
Assume $\A(g)=0$.

Then, with the notations of proposition \ref{geo-ineq},
\begin{eqnarray*}
\AVR(g)=\sum_{i\in I}\frac{\omega_n}{\arrowvert \Gamma_i\arrowvert }.
\end{eqnarray*}

\end{coro}

Of course, this corollary is weaker than corollary \ref{flat} but it can be proved directly as above.

\begin{rk}
Note that when $n$ is odd, $\arrowvert\Gamma\arrowvert=1$, since the hypersurfaces are orientable. Moreover, in this case, one can assume only $\A(g)<1$ to get the same result by using a "light" version of the Klingenberg's theorem (which is also due to him) which asserts that "any orientable compact  even-dimensional Riemannian manifold $(N,h)$ with sectional curvature in $(0,1]$ has $\inj(N,h)\geq\pi$". 
\end{rk}

\begin{rk}
The assumption $n\geq 3$ is sharp in the following sense: there exists a complete two-dimensional expanding gradient soliton with nonnegative scalar curvature, asymptotically flat (i.e. $\A(g)=0$) such that $\AVR(g)<\omega_2$, see chap.4, section 5 of \cite{Cho-Lu}.
\end{rk}

\begin{rk}
Note that these inequalities do not depend on the geometry of $f$.
Thus, are these inequalities more universal? For instance,
do they hold for Riemannian manifold with nonnegative Ricci curvature, positive $\AVR$, and finite $A$? This will be the subject of forthcoming papers.
\end{rk}

At this stage, one can ask if there is some rigidity results concerning the asymptotic curvature ratio $\A(g)$ of a nonnegatively curved expanding gradient Ricci soliton. In fact, Huai-dong Cao has built a $1$-parameter family of expanding gradient Ricci soliton with nonnegative sectional curvature: see \cite{Cao-Lim}. These examples are rotationnally symmetric, they behave at infinity like a metric cone and their asymptotic curvature ratios take any values in $(0,+\infty)$.

\bibliographystyle{alpha}
\bibliography{bib}
\end{document}